\documentclass[12pt,preprint]{amsart}

\usepackage{amssymb,amsfonts,amsthm,amsmath,amscd}

\usepackage{enumerate}
\usepackage{verbatim}

\usepackage{pgf,tikz}
\usetikzlibrary{arrows}

\theoremstyle{plain}
\newtheorem{theorem}{Theorem}
\newtheorem{proposition}{Proposition}

\newtheorem{corollary}{Corollary}

\theoremstyle{definition}
\newtheorem{definition}{Definition}

\newtheorem{remark}{Remark}

\begin{document}

\title{Equitable Candy Sharing}

\author{Grant Cairns}

\maketitle

\begin{abstract}
Children, sitting in a circle, each have a nonnegative number of candies in front of them. A whistle is blown and each child with more than one candy passes one candy to the left and one to the right. The sharing process is repeated until a fixed state is attained, or the system enters a periodic cycle. This paper treats the case where  the total number of candies equals the number of children. For a given initial distribution of candies, a necessary and sufficient condition is given for the system to ultimately attain the equitable distribution in which each child has one candy.\end{abstract}

\section{Introduction.}

In the simplest form of \emph{candy sharing} games, children sit in a circle, and they each initially have a nonnegative number of candies in front of them, the number possibly differing from child to child. Then a whistle is blown and the children pass some of the candies to their immediate neighbours, to the left and to the right. The number of candies that each child passes depends only on how many candies they have, and differs according to the variant of the game. The process is repeated several times; at each stage, the whistle is blown and the candies shared. Since the system is a finite one, ultimately either the process terminates in a fixed state (e.g., the children all have the same number of candies), or the system enters a cycle of some finite length greater than one.
The basic problem is to determine which of these two outcomes is attained, given the initial distribution of candies.
Variations of the candy sharing game have been used in competitions, and extension activities for students, dating back at least to 1963; as mentioned in \cite{IT}. 

In this paper, we consider one of the commonly studied candy sharing games; see \cite[Appendix VI]{T}. It has the following sharing rule:

\smallskip
\emph{Each child with more than one candy passes one candy to the left and one to the right}.
\smallskip

\noindent
The game with $n$ of children and $m$ candies terminates in a fixed state when $m<n $  \cite[Appendix VI]{T} and when $m\geq 3 n$ \cite{KK}.  In this paper, we examine the case $m=n$. We will call this the \emph{balanced candy sharing game}. In this case, the only possible fixed state is where each child has one candy. To give an example of the kind of cycle that can occur, number the children cyclically in clockwise order from 1 to $n$, and for each $i$ let $c_i$ denote the number of candies held by child $i$. So the string $(c_1,\dots,c_n)$ represents the state of the system. Consider the state
\[
W=(0,2,\underbrace{1,1,\dots,1}_{(n-2)\, 1\text{s}}).
\]
Given any state $S$, let $f(S)$ denote the state after the next iteration of the game. So
\[
f(W)=(1,0,2,\underbrace{1,1,\dots,1}_{(n-3)\,1\text{s}})
\quad \text{and}\quad f^2(W)=(1,1,0,2,\underbrace{1,1,\dots,1}_{(n-4)\, 1\text{s}}),
\]
and so forth. Thus the iterates of $W$ form a travelling wave that moves clockwise around the circle; see Figure \ref{fig}.  

\begin{figure}
\begin{center}
\begin{tikzpicture}[scale=0.35,auto];
\pgfmathsetmacro {\r}{3.3};
\pgfmathsetmacro {\t}{3.8};
\draw[thick,red] (0,0) circle (3);

\pgfmathsetmacro {\a}{330};

\draw[rotate=\a,fill,blue] (0,\t) ellipse (.3 and .2);
\draw[rotate=\a-4,fill] (0,\t)--(.2,\t+.2)--(.2,\t-.2);
\draw[rotate=\a+4,fill] (0,\t)--(-.2,\t+.2)--(-.2,\t-.2);

\pgfmathsetmacro {\a}{30};

\draw[rotate=\a,fill,blue] (0,\r) ellipse (.3 and .2);
\draw[rotate=\a-5,fill] (0,\r)--(.2,\r+.2)--(.2,\r-.2);
\draw[rotate=\a+5,fill] (0,\r)--(-.2,\r+.2)--(-.2,\r-.2);

\pgfmathsetmacro {\a}{60};

\draw[rotate=\a,fill,blue] (0,\r) ellipse (.3 and .2);
\draw[rotate=\a-5,fill] (0,\r)--(.2,\r+.2)--(.2,\r-.2);
\draw[rotate=\a+5,fill] (0,\r)--(-.2,\r+.2)--(-.2,\r-.2);

\pgfmathsetmacro {\a}{90};

\draw[rotate=\a,fill,blue] (0,\r) ellipse (.3 and .2);
\draw[rotate=\a-5,fill] (0,\r)--(.2,\r+.2)--(.2,\r-.2);
\draw[rotate=\a+5,fill] (0,\r)--(-.2,\r+.2)--(-.2,\r-.2);

\pgfmathsetmacro {\a}{120};

\draw[rotate=\a,fill,blue] (0,\r) ellipse (.3 and .2);
\draw[rotate=\a-5,fill] (0,\r)--(.2,\r+.2)--(.2,\r-.2);
\draw[rotate=\a+5,fill] (0,\r)--(-.2,\r+.2)--(-.2,\r-.2);

\pgfmathsetmacro {\a}{150};

\draw[rotate=\a,fill,blue] (0,\r) ellipse (.3 and .2);
\draw[rotate=\a-5,fill] (0,\r)--(.2,\r+.2)--(.2,\r-.2);
\draw[rotate=\a+5,fill] (0,\r)--(-.2,\r+.2)--(-.2,\r-.2);

\pgfmathsetmacro {\a}{180};

\draw[rotate=\a,fill,blue] (0,\r) ellipse (.3 and .2);
\draw[rotate=\a-5,fill] (0,\r)--(.2,\r+.2)--(.2,\r-.2);
\draw[rotate=\a+5,fill] (0,\r)--(-.2,\r+.2)--(-.2,\r-.2);

\pgfmathsetmacro {\a}{210};

\draw[rotate=\a,fill,blue] (0,\r) ellipse (.3 and .2);
\draw[rotate=\a-5,fill] (0,\r)--(.2,\r+.2)--(.2,\r-.2);
\draw[rotate=\a+5,fill] (0,\r)--(-.2,\r+.2)--(-.2,\r-.2);

\pgfmathsetmacro {\a}{240};

\draw[rotate=\a,fill,blue] (0,\r) ellipse (.3 and .2);
\draw[rotate=\a-5,fill] (0,\r)--(.2,\r+.2)--(.2,\r-.2);
\draw[rotate=\a+5,fill] (0,\r)--(-.2,\r+.2)--(-.2,\r-.2);

\pgfmathsetmacro {\a}{270};

\draw[rotate=\a,fill,blue] (0,\r) ellipse (.3 and .2);
\draw[rotate=\a-5,fill] (0,\r)--(.2,\r+.2)--(.2,\r-.2);
\draw[rotate=\a+5,fill] (0,\r)--(-.2,\r+.2)--(-.2,\r-.2);

\pgfmathsetmacro {\a}{300};

\draw[rotate=\a,fill,blue] (0,\r) ellipse (.3 and .2);
\draw[rotate=\a-5,fill] (0,\r)--(.2,\r+.2)--(.2,\r-.2);
\draw[rotate=\a+5,fill] (0,\r)--(-.2,\r+.2)--(-.2,\r-.2);

\pgfmathsetmacro {\a}{330};

\draw[rotate=\a,fill,blue] (0,\r) ellipse (.3 and .2);
\draw[rotate=\a-5,fill] (0,\r)--(.2,\r+.2)--(.2,\r-.2);
\draw[rotate=\a+5,fill] (0,\r)--(-.2,\r+.2)--(-.2,\r-.2);

\end{tikzpicture}
\hskip.7cm
\begin{tikzpicture}[scale=0.35,auto];
\pgfmathsetmacro {\r}{3.3};
\pgfmathsetmacro {\t}{3.8};
\draw[thick,red] (0,0) circle (3);

\pgfmathsetmacro {\a}{0};

\draw[rotate=\a,fill,blue] (0,\r) ellipse (.3 and .2);
\draw[rotate=\a-5,fill] (0,\r)--(.2,\r+.2)--(.2,\r-.2);
\draw[rotate=\a+5,fill] (0,\r)--(-.2,\r+.2)--(-.2,\r-.2);

\pgfmathsetmacro {\a}{30};

\draw[rotate=\a,fill,blue] (0,\r) ellipse (.3 and .2);
\draw[rotate=\a-5,fill] (0,\r)--(.2,\r+.2)--(.2,\r-.2);
\draw[rotate=\a+5,fill] (0,\r)--(-.2,\r+.2)--(-.2,\r-.2);

\pgfmathsetmacro {\a}{60};

\draw[rotate=\a,fill,blue] (0,\r) ellipse (.3 and .2);
\draw[rotate=\a-5,fill] (0,\r)--(.2,\r+.2)--(.2,\r-.2);
\draw[rotate=\a+5,fill] (0,\r)--(-.2,\r+.2)--(-.2,\r-.2);

\pgfmathsetmacro {\a}{90};

\draw[rotate=\a,fill,blue] (0,\r) ellipse (.3 and .2);
\draw[rotate=\a-5,fill] (0,\r)--(.2,\r+.2)--(.2,\r-.2);
\draw[rotate=\a+5,fill] (0,\r)--(-.2,\r+.2)--(-.2,\r-.2);

\pgfmathsetmacro {\a}{120};

\draw[rotate=\a,fill,blue] (0,\r) ellipse (.3 and .2);
\draw[rotate=\a-5,fill] (0,\r)--(.2,\r+.2)--(.2,\r-.2);
\draw[rotate=\a+5,fill] (0,\r)--(-.2,\r+.2)--(-.2,\r-.2);

\pgfmathsetmacro {\a}{150};

\draw[rotate=\a,fill,blue] (0,\r) ellipse (.3 and .2);
\draw[rotate=\a-5,fill] (0,\r)--(.2,\r+.2)--(.2,\r-.2);
\draw[rotate=\a+5,fill] (0,\r)--(-.2,\r+.2)--(-.2,\r-.2);

\pgfmathsetmacro {\a}{180};

\draw[rotate=\a,fill,blue] (0,\r) ellipse (.3 and .2);
\draw[rotate=\a-5,fill] (0,\r)--(.2,\r+.2)--(.2,\r-.2);
\draw[rotate=\a+5,fill] (0,\r)--(-.2,\r+.2)--(-.2,\r-.2);

\pgfmathsetmacro {\a}{210};

\draw[rotate=\a,fill,blue] (0,\r) ellipse (.3 and .2);
\draw[rotate=\a-5,fill] (0,\r)--(.2,\r+.2)--(.2,\r-.2);
\draw[rotate=\a+5,fill] (0,\r)--(-.2,\r+.2)--(-.2,\r-.2);

\pgfmathsetmacro {\a}{240};

\draw[rotate=\a,fill,blue] (0,\r) ellipse (.3 and .2);
\draw[rotate=\a-5,fill] (0,\r)--(.2,\r+.2)--(.2,\r-.2);
\draw[rotate=\a+5,fill] (0,\r)--(-.2,\r+.2)--(-.2,\r-.2);

\pgfmathsetmacro {\a}{270};

\draw[rotate=\a,fill,blue] (0,\r) ellipse (.3 and .2);
\draw[rotate=\a-5,fill] (0,\r)--(.2,\r+.2)--(.2,\r-.2);
\draw[rotate=\a+5,fill] (0,\r)--(-.2,\r+.2)--(-.2,\r-.2);

\pgfmathsetmacro {\a}{300};

\draw[rotate=\a,fill,blue] (0,\r) ellipse (.3 and .2);
\draw[rotate=\a-5,fill] (0,\r)--(.2,\r+.2)--(.2,\r-.2);
\draw[rotate=\a+5,fill] (0,\r)--(-.2,\r+.2)--(-.2,\r-.2);

\pgfmathsetmacro {\a}{300};

\draw[rotate=\a,fill,blue] (0,\t) ellipse (.3 and .2);
\draw[rotate=\a-4,fill] (0,\t)--(.2,\t+.2)--(.2,\t-.2);
\draw[rotate=\a+4,fill] (0,\t)--(-.2,\t+.2)--(-.2,\t-.2);

\end{tikzpicture}
\hskip.7cm
\begin{tikzpicture}[scale=0.35,auto];
\pgfmathsetmacro {\r}{3.3};
\pgfmathsetmacro {\t}{3.8};
\draw[thick,red] (0,0) circle (3);

\pgfmathsetmacro {\a}{0};

\draw[rotate=\a,fill,blue] (0,\r) ellipse (.3 and .2);
\draw[rotate=\a-5,fill] (0,\r)--(.2,\r+.2)--(.2,\r-.2);
\draw[rotate=\a+5,fill] (0,\r)--(-.2,\r+.2)--(-.2,\r-.2);

\pgfmathsetmacro {\a}{30};

\draw[rotate=\a,fill,blue] (0,\r) ellipse (.3 and .2);
\draw[rotate=\a-5,fill] (0,\r)--(.2,\r+.2)--(.2,\r-.2);
\draw[rotate=\a+5,fill] (0,\r)--(-.2,\r+.2)--(-.2,\r-.2);

\pgfmathsetmacro {\a}{60};

\draw[rotate=\a,fill,blue] (0,\r) ellipse (.3 and .2);
\draw[rotate=\a-5,fill] (0,\r)--(.2,\r+.2)--(.2,\r-.2);
\draw[rotate=\a+5,fill] (0,\r)--(-.2,\r+.2)--(-.2,\r-.2);

\pgfmathsetmacro {\a}{90};

\draw[rotate=\a,fill,blue] (0,\r) ellipse (.3 and .2);
\draw[rotate=\a-5,fill] (0,\r)--(.2,\r+.2)--(.2,\r-.2);
\draw[rotate=\a+5,fill] (0,\r)--(-.2,\r+.2)--(-.2,\r-.2);

\pgfmathsetmacro {\a}{120};

\draw[rotate=\a,fill,blue] (0,\r) ellipse (.3 and .2);
\draw[rotate=\a-5,fill] (0,\r)--(.2,\r+.2)--(.2,\r-.2);
\draw[rotate=\a+5,fill] (0,\r)--(-.2,\r+.2)--(-.2,\r-.2);

\pgfmathsetmacro {\a}{150};

\draw[rotate=\a,fill,blue] (0,\r) ellipse (.3 and .2);
\draw[rotate=\a-5,fill] (0,\r)--(.2,\r+.2)--(.2,\r-.2);
\draw[rotate=\a+5,fill] (0,\r)--(-.2,\r+.2)--(-.2,\r-.2);

\pgfmathsetmacro {\a}{180};

\draw[rotate=\a,fill,blue] (0,\r) ellipse (.3 and .2);
\draw[rotate=\a-5,fill] (0,\r)--(.2,\r+.2)--(.2,\r-.2);
\draw[rotate=\a+5,fill] (0,\r)--(-.2,\r+.2)--(-.2,\r-.2);

\pgfmathsetmacro {\a}{210};

\draw[rotate=\a,fill,blue] (0,\r) ellipse (.3 and .2);
\draw[rotate=\a-5,fill] (0,\r)--(.2,\r+.2)--(.2,\r-.2);
\draw[rotate=\a+5,fill] (0,\r)--(-.2,\r+.2)--(-.2,\r-.2);

\pgfmathsetmacro {\a}{240};

\draw[rotate=\a,fill,blue] (0,\r) ellipse (.3 and .2);
\draw[rotate=\a-5,fill] (0,\r)--(.2,\r+.2)--(.2,\r-.2);
\draw[rotate=\a+5,fill] (0,\r)--(-.2,\r+.2)--(-.2,\r-.2);

\pgfmathsetmacro {\a}{270};

\draw[rotate=\a,fill,blue] (0,\t) ellipse (.3 and .2);
\draw[rotate=\a-4,fill] (0,\t)--(.2,\t+.2)--(.2,\t-.2);
\draw[rotate=\a+4,fill] (0,\t)--(-.2,\t+.2)--(-.2,\t-.2);

\pgfmathsetmacro {\a}{270};

\draw[rotate=\a,fill,blue] (0,\r) ellipse (.3 and .2);
\draw[rotate=\a-5,fill] (0,\r)--(.2,\r+.2)--(.2,\r-.2);
\draw[rotate=\a+5,fill] (0,\r)--(-.2,\r+.2)--(-.2,\r-.2);

\pgfmathsetmacro {\a}{330};

\draw[rotate=\a,fill,blue] (0,\r) ellipse (.3 and .2);
\draw[rotate=\a-5,fill] (0,\r)--(.2,\r+.2)--(.2,\r-.2);
\draw[rotate=\a+5,fill] (0,\r)--(-.2,\r+.2)--(-.2,\r-.2);

\end{tikzpicture}
\hskip.7cm
\begin{tikzpicture}[scale=0.35,auto];
\pgfmathsetmacro {\r}{3.3};
\pgfmathsetmacro {\t}{3.8};
\draw[thick,red] (0,0) circle (3);

\pgfmathsetmacro {\a}{0};

\draw[rotate=\a,fill,blue] (0,\r) ellipse (.3 and .2);
\draw[rotate=\a-5,fill] (0,\r)--(.2,\r+.2)--(.2,\r-.2);
\draw[rotate=\a+5,fill] (0,\r)--(-.2,\r+.2)--(-.2,\r-.2);

\pgfmathsetmacro {\a}{30};

\draw[rotate=\a,fill,blue] (0,\r) ellipse (.3 and .2);
\draw[rotate=\a-5,fill] (0,\r)--(.2,\r+.2)--(.2,\r-.2);
\draw[rotate=\a+5,fill] (0,\r)--(-.2,\r+.2)--(-.2,\r-.2);

\pgfmathsetmacro {\a}{60};

\draw[rotate=\a,fill,blue] (0,\r) ellipse (.3 and .2);
\draw[rotate=\a-5,fill] (0,\r)--(.2,\r+.2)--(.2,\r-.2);
\draw[rotate=\a+5,fill] (0,\r)--(-.2,\r+.2)--(-.2,\r-.2);

\pgfmathsetmacro {\a}{90};

\draw[rotate=\a,fill,blue] (0,\r) ellipse (.3 and .2);
\draw[rotate=\a-5,fill] (0,\r)--(.2,\r+.2)--(.2,\r-.2);
\draw[rotate=\a+5,fill] (0,\r)--(-.2,\r+.2)--(-.2,\r-.2);

\pgfmathsetmacro {\a}{120};

\draw[rotate=\a,fill,blue] (0,\r) ellipse (.3 and .2);
\draw[rotate=\a-5,fill] (0,\r)--(.2,\r+.2)--(.2,\r-.2);
\draw[rotate=\a+5,fill] (0,\r)--(-.2,\r+.2)--(-.2,\r-.2);

\pgfmathsetmacro {\a}{150};

\draw[rotate=\a,fill,blue] (0,\r) ellipse (.3 and .2);
\draw[rotate=\a-5,fill] (0,\r)--(.2,\r+.2)--(.2,\r-.2);
\draw[rotate=\a+5,fill] (0,\r)--(-.2,\r+.2)--(-.2,\r-.2);

\pgfmathsetmacro {\a}{180};

\draw[rotate=\a,fill,blue] (0,\r) ellipse (.3 and .2);
\draw[rotate=\a-5,fill] (0,\r)--(.2,\r+.2)--(.2,\r-.2);
\draw[rotate=\a+5,fill] (0,\r)--(-.2,\r+.2)--(-.2,\r-.2);

\pgfmathsetmacro {\a}{210};

\draw[rotate=\a,fill,blue] (0,\r) ellipse (.3 and .2);
\draw[rotate=\a-5,fill] (0,\r)--(.2,\r+.2)--(.2,\r-.2);
\draw[rotate=\a+5,fill] (0,\r)--(-.2,\r+.2)--(-.2,\r-.2);

\pgfmathsetmacro {\a}{240};

\draw[rotate=\a,fill,blue] (0,\r) ellipse (.3 and .2);
\draw[rotate=\a-5,fill] (0,\r)--(.2,\r+.2)--(.2,\r-.2);
\draw[rotate=\a+5,fill] (0,\r)--(-.2,\r+.2)--(-.2,\r-.2);

\pgfmathsetmacro {\a}{240};

\draw[rotate=\a,fill,blue] (0,\t) ellipse (.3 and .2);
\draw[rotate=\a-4,fill] (0,\t)--(.2,\t+.2)--(.2,\t-.2);
\draw[rotate=\a+4,fill] (0,\t)--(-.2,\t+.2)--(-.2,\t-.2);

\pgfmathsetmacro {\a}{300};

\draw[rotate=\a,fill,blue] (0,\r) ellipse (.3 and .2);
\draw[rotate=\a-5,fill] (0,\r)--(.2,\r+.2)--(.2,\r-.2);
\draw[rotate=\a+5,fill] (0,\r)--(-.2,\r+.2)--(-.2,\r-.2);

\pgfmathsetmacro {\a}{330};

\draw[rotate=\a,fill,blue] (0,\r) ellipse (.3 and .2);
\draw[rotate=\a-5,fill] (0,\r)--(.2,\r+.2)--(.2,\r-.2);
\draw[rotate=\a+5,fill] (0,\r)--(-.2,\r+.2)--(-.2,\r-.2);

\end{tikzpicture}
\end{center}
\caption{Clockwise travelling wave of candies}\label{fig}
\end{figure}
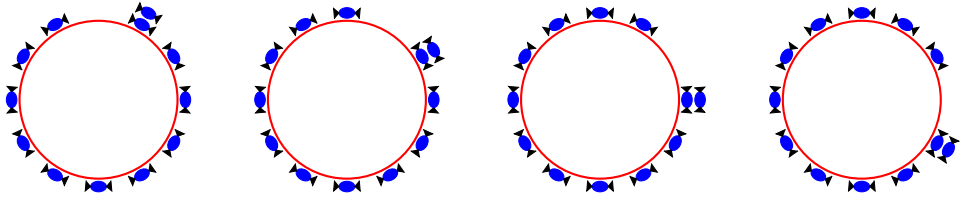

Recall that a state $S$ is \emph{periodic}  if $f^i(S)=S$, for some natural number $i$, and the \emph{period} of $S$ is the smallest such $i$. So the fixed state $I$ is periodic, with (albeit trivial) period $1$. The periodic state $W$ given above has period $n$.

To give slightly more complicated examples, let us introduce some more notation. Let 
\[
P=(0,2), \quad \bar P=(2,0), \quad I=(1),
\]
Furthermore, for strings $A,B$, let us denote by $AB$ the string obtained by concatenating $A$ and $B$. For example, $W=PI^{n-2}$. Consider natural numbers $i_1,i_2, \dots, i_\ell$ and $j_1,j_2, \dots, j_\ell$ with $2(i_1+i_2+ \dots+ i_\ell)+j_1+j_2+ \dots+ j_\ell=n$. Notice that the strings
$P^{i_1} I^{j_i} P^{i_2} I^{j_2}  \dots P^{i_\ell} I^{j_\ell}$
travel clockwise around the circle. They are periodic, where the  period is some divisor of $n$. The strings
of the form
$\bar P^{i_1} I^{j_i} \bar P^{i_2} I^{j_2}  \dots \bar P^{i_\ell} I^{j_\ell}$ have the same property, but travel anti-clockwise around the circle. When $n$ is even, say $n=2k$, the strings $P^k$ and $\bar P^k=f(P^k)$ have period 2.
We claim that this is effectively the complete list of periodic states.

\begin{theorem}\label{th}
Up to cyclic rotations, the only periodic states of the balanced candy sharing game with $n$ children are:
\begin{enumerate}[\rm(a)]
\item[(a)]  the strings of the form
$P^{i_1} I^{j_i} P^{i_2} I^{j_2}  \dots P^{i_\ell} I^{j_\ell}$, where $j_1,j_2,\dots,j_\ell>0$,
\item[(b)]  the strings of the form $\bar P^{i_1} I^{j_i} \bar P^{i_2} I^{j_2}  \dots \bar P^{i_\ell} I^{j_\ell}$, where $j_1,j_2,\dots,j_\ell>0$,
\item[(c)]  the string $P^{\frac{n}2}$ (which can only occur when $n$ is even),  
\item[(d)] the fixed state $I^n$. 
\end{enumerate}
\end{theorem}

The proof of Theorem~\ref{th} employs the ideas of \cite[pp.~252--253]{T}. In particular, it uses the notion of the \emph{index} of a state $S=(c_1,\dots,c_n)$.
Consider the substrings $(c_i,c_{i+1},\dots,c_{i+k-1})$ of $S$ of length $k$, where $1\leq k \leq n $, and here and throughout the paper, the subscripts are taken modulo $n$. The \emph{deficiency} of such a substring is defined to be
\[
\begin{cases}k-(c_i+c_{i+1}+\dots+c_{i+k-1}) &:\ \text{if}\ c_i+c_{i+1}+\dots+c_{i+k-1}\leq k\\
0&:\ \text{otherwise}.
\end{cases}
\]
Then the \emph{index} of $S$ is defined to be the sum of the deficiencies of all substrings of $S$. The key property of the index is that with each iteration of the game, the index either decreases or remains unchanged (in the context examined in \cite{T}, the index always decreased). To see this, suppose that in the state $S$, there are $i$ children with more than one candy each. It is convenient to consider the effect on the index of these $i$ children sharing one at a time. As the sharing takes place, the change in the index may temporarily depend on the order in which the children share, but the final outcome state, $f(S)$, and its index, do not depend on the order. So it suffices to show that the effect of one child sharing is to either decrease or maintain the index.  Without loss of generality, let us assume child number 1 shares; so $c_1\geq 2$. We will consider the effect on the deficiencies of all the strings. First note that if a string includes  $c_1,c_2$ and $c_n$, then sharing by child 1 will have no effect on the deficiency of the string. Similarly,  if a string doesn't include any of   $c_1,c_2$ or $c_n$, then sharing by child 1 will have no effect on the deficiency of the string. 

Now consider the string $(c_1)$ itself. After sharing, this becomes $(c_1-2)$. So the deficiency of this string only increases if $c_1=2$, and in this case the deficiency of this string increases by $1$. Consider the complementary string $(c_2,c_3,\dots,c_n)$. It has deficiency $c_1-1$, which is positive. After sharing, the string becomes $(c_2+1,c_3,\dots,c_n+1)$, which has deficiency $c_1-3$ if $c_1>3$, and $0$ otherwise. So together the deficiency of the pair $(c_1)$, $(c_2,c_3,\dots,c_n)$ decreases if $c_1>2$, and it remains constant if and only if $c_1=2$. 

Next we consider strings that include $c_1$ and either $c_2$ or $c_n$, but not both.  Without loss of generality, we need only consider strings of the form $(c_1,c_2,\dots,c_k)$, for some $k$ with $1<k<n$. When child 1 shares, the deficiency of the string increases by 1 if $c_1+c_{2}+\dots+c_{k}\leq k$, and remains zero otherwise. Consider the string  $(c_2,c_3,\dots,c_k)$. When child 1 shares, the deficiency of this string decreases by 1 if $c_2+c_{3}+\dots+c_{k}\leq k-2$, and remains zero otherwise. Notice that since $c_1\geq 2$, we have
\[
c_1+c_{2}+\dots+c_{k}\leq k \implies c_2+c_{3}+\dots+c_{k}\leq k-2.
\]
So together the deficiency of the pair $(c_1,c_2,\dots,c_k)$, $(c_2,c_3,\dots,c_k)$ either decreases or remains constant, and it 
remains constant if and only if $c_1=2$. 

We have considered all possible strings, so let us summarize the conclusions of the above discussion.

\begin{proposition}\label{Prop} When one child shares their candies, the index either decreases or remains constant, and it remains constant only when the child who is sharing has exactly 2 candies.
\end{proposition}

\begin{corollary}\label{Cor} At each iteration of the game, the index either decreases or remains constant. Moreover, if the index of a state $S$ equals the index of the next state $f(S)$, then the maximum element of $S$ is 2.
\end{corollary}

We draw the reader's attention to a subtlety: the above discussion does not imply that if the maximum element of a state $S$ is 2, then the index of the next state $f(S)$ equals the index of $S$. Consider the case where $n=4$ and let $S=(2,2,0,0)$. This state has index 6; there are 4 strings of deficiency 1, and one string deficiency 2. One has $f(S)=(1,1,1,1)$, which has index 0. If child 1 shares first, the result obtained from state $S$ is the state $S'=(0,3,0,1)$. This state has index 6, the same as the index of $S$, as it must be by Proposition \ref{Prop}. But while the maximum element in $S$ is 2, the state $S'$ has maximum element 3. Continuing from $S'$, if child 2 now shares, the result is the state $f(S)=(1,1,1,1)$. See Figure \ref{fig2}.

\begin{figure}[h]
\begin{center}
\begin{tikzpicture}[scale=0.45,auto];
\pgfmathsetmacro {\r}{2.3};
\pgfmathsetmacro {\t}{2.8};
\pgfmathsetmacro {\s}{3.3};
\draw[thick,red] (0,0) circle (2);

\pgfmathsetmacro {\a}{0};

\draw[rotate=\a,fill,blue] (0,\r) ellipse (.35 and .2);
\draw[rotate=\a-10,fill] (0,\r)--(.2,\r+.2)--(.2,\r-.2);
\draw[rotate=\a+10,fill] (0,\r)--(-.2,\r+.2)--(-.2,\r-.2);

\pgfmathsetmacro {\a}{90};

\draw[rotate=\a,fill,blue] (0,\r) ellipse (.35 and .2);
\draw[rotate=\a-10,fill] (0,\r)--(.2,\r+.2)--(.2,\r-.2);
\draw[rotate=\a+10,fill] (0,\r)--(-.2,\r+.2)--(-.2,\r-.2);

\pgfmathsetmacro {\a}{180};

\draw[rotate=\a,fill,blue] (0,\r) ellipse (.35 and .2);
\draw[rotate=\a-10,fill] (0,\r)--(.2,\r+.2)--(.2,\r-.2);
\draw[rotate=\a+10,fill] (0,\r)--(-.2,\r+.2)--(-.2,\r-.2);

\pgfmathsetmacro {\a}{270};

\draw[rotate=\a,fill,blue] (0,\r) ellipse (.35 and .2);
\draw[rotate=\a-10,fill] (0,\r)--(.2,\r+.2)--(.2,\r-.2);
\draw[rotate=\a+10,fill] (0,\r)--(-.2,\r+.2)--(-.2,\r-.2);

\pgfmathsetmacro {\cx}{6};
\pgfmathsetmacro {\cy}{6*sqrt(3)};
\draw[thick,red] (\cx,\cy) circle (2);

\pgfmathsetmacro {\a}{90};
\draw[rotate around ={\a:(\cx,\cy)},fill,blue] (\cx,\cy+\r) ellipse (.35 and .2);
\draw[rotate around ={\a-10:(\cx,\cy)},fill,blue]  (\cx,\cy+\r)--(\cx+.2,\cy+\r+.2)--(\cx+.2,\cy+\r-.2);
\draw[rotate around ={\a+10:(\cx,\cy)},fill,blue] (\cx,\cy+\r)--(\cx+-.2,\cy+\r+.2)--(\cx+-.2,\cy+\r-.2);

\pgfmathsetmacro {\a}{-90};
\draw[rotate around ={\a:(\cx,\cy)},fill,blue] (\cx,\cy+\r) ellipse (.35 and .2);
\draw[rotate around ={\a-10:(\cx,\cy)},fill,blue]  (\cx,\cy+\r)--(\cx+.2,\cy+\r+.2)--(\cx+.2,\cy+\r-.2);
\draw[rotate around ={\a+10:(\cx,\cy)},fill,blue] (\cx,\cy+\r)--(\cx+-.2,\cy+\r+.2)--(\cx+-.2,\cy+\r-.2);

\pgfmathsetmacro {\a}{-90};
\draw[rotate around ={\a:(\cx,\cy)},fill,blue] (\cx,\cy+\t) ellipse (.35 and .2);
\draw[rotate around ={\a-8:(\cx,\cy)},fill,blue]  (\cx,\cy+\t)--(\cx+.2,\cy+\t+.2)--(\cx+.2,\cy+\t-.2);
\draw[rotate around ={\a+8:(\cx,\cy)},fill,blue] (\cx,\cy+\t)--(\cx+-.2,\cy+\t+.2)--(\cx+-.2,\cy+\t-.2);

\pgfmathsetmacro {\a}{-90};
\draw[rotate around ={\a:(\cx,\cy)},fill,blue] (\cx,\cy+\s) ellipse (.35 and .2);
\draw[rotate around ={\a-6.5:(\cx,\cy)},fill,blue]  (\cx,\cy+\s)--(\cx+.2,\cy+\s+.2)--(\cx+.2,\cy+\s-.2);
\draw[rotate around ={\a+6.5:(\cx,\cy)},fill,blue] (\cx,\cy+\s)--(\cx+-.2,\cy+\s+.2)--(\cx+-.2,\cy+\s-.2);

\pgfmathsetmacro {\cx}{-6};
\pgfmathsetmacro {\cy}{6*sqrt(3)};
\draw[thick,red] (\cx,\cy) circle (2);

\pgfmathsetmacro {\a}{0};
\draw[rotate around ={\a:(\cx,\cy)},fill,blue] (\cx,\cy+\r) ellipse (.35 and .2);
\draw[rotate around ={\a-10:(\cx,\cy)},fill,blue]  (\cx,\cy+\r)--(\cx+.2,\cy+\r+.2)--(\cx+.2,\cy+\r-.2);
\draw[rotate around ={\a+10:(\cx,\cy)},fill,blue] (\cx,\cy+\r)--(\cx+-.2,\cy+\r+.2)--(\cx+-.2,\cy+\r-.2);

\pgfmathsetmacro {\a}{0};
\draw[rotate around ={\a:(\cx,\cy)},fill,blue] (\cx,\cy+\t) ellipse (.35 and .2);
\draw[rotate around ={\a-8:(\cx,\cy)},fill,blue]  (\cx,\cy+\t)--(\cx+.2,\cy+\t+.2)--(\cx+.2,\cy+\t-.2);
\draw[rotate around ={\a+8:(\cx,\cy)},fill,blue] (\cx,\cy+\t)--(\cx+-.2,\cy+\t+.2)--(\cx+-.2,\cy+\t-.2);

\pgfmathsetmacro {\a}{-90};
\draw[rotate around ={\a:(\cx,\cy)},fill,blue] (\cx,\cy+\r) ellipse (.35 and .2);
\draw[rotate around ={\a-10:(\cx,\cy)},fill,blue]  (\cx,\cy+\r)--(\cx+.2,\cy+\r+.2)--(\cx+.2,\cy+\r-.2);
\draw[rotate around ={\a+10:(\cx,\cy)},fill,blue] (\cx,\cy+\r)--(\cx+-.2,\cy+\r+.2)--(\cx+-.2,\cy+\r-.2);

\pgfmathsetmacro {\a}{-90};
\draw[rotate around ={\a:(\cx,\cy)},fill,blue] (\cx,\cy+\t) ellipse (.35 and .2);
\draw[rotate around ={\a-8:(\cx,\cy)},fill,blue]  (\cx,\cy+\t)--(\cx+.2,\cy+\t+.2)--(\cx+.2,\cy+\t-.2);
\draw[rotate around ={\a+8:(\cx,\cy)},fill,blue] (\cx,\cy+\t)--(\cx+-.2,\cy+\t+.2)--(\cx+-.2,\cy+\t-.2);

   \draw (-4,{4*sqrt(3)})  edge[->,thick] node[anchor=east] {$f$} (-2,{2*sqrt(3)});
   \draw (4,{4*sqrt(3)})  edge[->,thick]  (2,{2*sqrt(3)});
   \draw (-2,{6*sqrt(3)})  edge[->,thick]  (2,{6*sqrt(3)});

\end{tikzpicture}
\end{center}
\caption{An index lowering example}\label{fig2}
\end{figure}
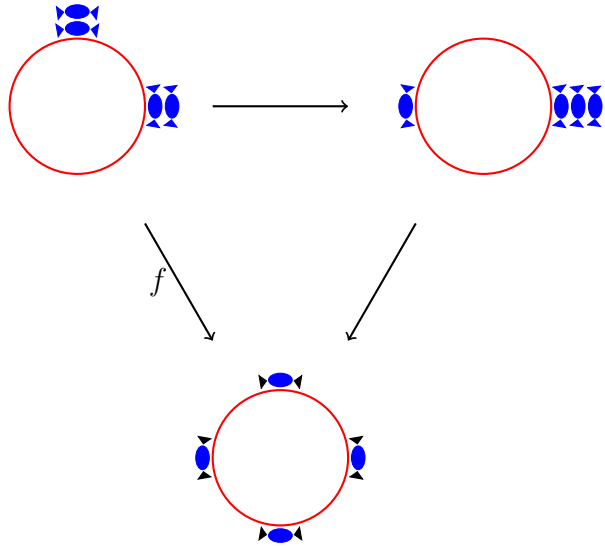

\begin{proof}[Proof of Theorem~\ref{th}] Suppose that $S=(c_1,\dots,c_n)$ is a periodic state. Since the index of $S$ is a nonnegative integer, and it never increases at any step in the game, it must be that the index is constant on the cycle defined by $S$. So, by the above corollary,  the maximum element of $S$ is 2.

Notice that there cannot be consecutive 2s in $S$. Indeed, arguing by contradiction, suppose that $c_i=c_{i+1}=2$ and consider what happens if only child $i$ shares. By the above proposition, this doesn't change the index, but now the $(i+1)$-st element is 3. Then if child $i+1$ shares, the index decreases, again by the above proposition. Consequently $f(S)$ has lower index than $S$, which is a contradiction.

Now notice that $S$ cannot contain a string equal to $(2,1,2)$. Indeed, if it did, $f(S)$ would contain a string of the form $(*,3,*)$, and so $f(f(S))$ would have lower index, by the above proposition.

 Since the sum of the elements of $S$ is $n$, the set $S$ contains the same number of 0s as 2s. We claim that the 0s and 2s alternate around the circle. Indeed, otherwise $S$ would contain a string of the form $(2)I^k (2)$, where the number $k$ of 1s separating the 2s is possibly zero. After one  iteration of the game, the string has the form $(*,2)I^{k-2} (2,*)$.
Further iterations would consequently ultimately lead either to two consecutive 2s, or a string equal to $(2,1,2)$, depending on whether $k$ is even or odd, respectively. But we have just seen that both of these possibilities are impossible.

Now notice that if a periodic  state has a 2 in position $i$, then because there are no consecutive 2s, after the next iteration of the game, there will be a 0 in position $i$. 
We use this observation to show that each $2$ is adjacent to a $0$. Indeed, otherwise $S$ would have a string equal to $(a_i,a_{i+1},a_{i+2})=(1,2,1)$, because there are no consecutive 2s. But since $S$ is periodic, there is a periodic state, $T=(b_1,b_2,\dots,b_n)$ say, with $f(T)=S$. But as we have just observed, if $b_k=2$ for some $k$, then $a_k=0$. Thus none of the entries $b_i,b_{i+1},b_{i+2}$ is equal to $2$; that is, they are each 0 or 1. But then, as $f(T)=S$, it is impossible that $a_{i+1}=2$. Thus each $2$ is adjacent to a $0$. Since there is the same number of 0s as 2s, we can also conclude that each $0$ is adjacent to a $2$.

Let us summarize our three main findings so far:
\begin{enumerate}
\item The periodic state $S$ consists only of 0s, 1s and 2s. 
\item The 0s and 2s alternate around the circle. 
\item Each 2 is adjacent to a 0, and each $0$ is adjacent to a $2$. 
\end{enumerate}

We are now ready to conclude the argument. If $S$ has no 2s, then $S$ consists entirely of 1s; that is, $S= I^n$. 
If $S$  has a 0 followed by a 2, then by cyclically permuting $S$, we may assume that $c_1=0, c_2=2$. Because the 0s and 2s alternate around the circle, $S$ has the form
$(0,2)I^{k_1}(0,*,\dots)$, for some $k_1\geq0$. So, since each 0 is adjacent to a 2, the state $S$ has the form
$(0,2)I^k(0,2,*,\dots)$. Thus, because the 0s and 2s alternate around the circle and each 0 is adjacent to a 2, $S$ has the form $(0,2)I^{k_1}(0,2)I^{k_2}(0,2)\dots$, for some $k_1,k_2\geq0$. Continuing in the same manner, we see that $S=(0,2)I^{k_1}(0,2)I^{k_2}\dots(0,2)I^{k_j}$, for some natural number $j$ and integers $k_1,k_2,\dots,k_j\geq 0$. Since some of the $k_i$ may be zero, we conclude that either $S=P^{\frac{n}2}$, or  $S=P^{i_1} I^{j_i} P^{i_2} I^{j_2}  \dots P^{i_\ell} I^{j_\ell}$, where $j_1,j_2,\dots,j_\ell>0$.

If $S$  has a 2 followed by a 0, then an entirely analogous argument shows that either $S=P^{\frac{n}2}$, or  $S=\bar P^{i_1} I^{j_i} \bar P^{i_2} I^{j_2}  \dots \bar P^{i_\ell} I^{j_\ell}$, where $j_1,j_2,\dots,j_\ell>0$.
\end{proof}

\section{Determining the outcome.}

For a particular starting state $S$, Theorem~\ref{Thsym} gives a finite number of possible periodic outcomes. 

\begin{definition}
We say that a state $S$ is \emph{clockwise biased},  \emph{anti-clockwise biased}, \emph{equivocal}, or \emph{ultimately equitable} if the game starting at $S$ results in a periodic state described  respectively in parts (a), (b), (c), (d) of Theorem~\ref{th}.
\end{definition}

The problem we examine now is to determine, for a given state $S$, whether  $S$ is clockwise biased, anti-clockwise biased,  equivocal or ultimately equitable.
For this purpose we introduce  an invariant. 

\begin{definition}
For a state $S=(a_1,a_2,\dots,a_n)$, let $\tau(S)$ be the element of the set $\{0,1,2,\dots,n-1\}$ defined by 
\[
\tau(S)\equiv n\frac{(-1)^n+1}4+\sum_{i=1}^n ia_i \pmod n.
\]
\end{definition}

\begin{remark}\label{Re}
Note that $\tau(I^n)= 0$, for all $n$. 
Indeed, if  $n$ is even, $n=2k$ say, then $\tau(I^n)\equiv k+\sum_{i=1}^n i= k+ k(2k+1)=2k(k+1)\equiv 0 \pmod n$. 
If  $n$ is odd, $n=2k+1$ say, then $\tau(I^n)\equiv\sum_{i=1}^n i= (k+1)(2k+1)\equiv 0 \pmod n$. 
\end{remark}

The effect of the rotational permutation $\sigma: (a_1,a_2,\dots,a_n) \mapsto (a_2,a_3,\dots,a_n,a_1)$ on $\tau(S)$ is as follows:
\begin{align*}
\tau(\sigma(S))&\equiv n\frac{(-1)^n+1}4+\sum_{i=1}^n ia_{i+1} =n\frac{(-1)^n+1}4+\sum_{i=1}^n (i-1)a_{i} \\
&\equiv\tau(S)-\sum_{i=1}^n a_{i} =\tau(S)- n \equiv \tau(S)\pmod n  .
\end{align*}
In particular,  $\tau(S)$ doesn't depend on where we start to number the children's positions. Furthermore, we claim  that the sharing process leaves $\tau(S)$ unchanged. 
To see this, since $\tau(S)$  is unchanged by cyclic rotations, it suffices to show that  $\tau(S)$ is unchanged if $a_2\geq 2$ and child 2 shares. After sharing, the state is $(a_1+1,a_2-2,a_3+1,a_4,\dots,a_n)$, and we have
\[
\tau(a_1+1,a_2-2,a_3+1,a_4,\dots,a_n)= \tau(S) +1-2\cdot 2+3=\tau(S).
\]

\begin{theorem}\label{Thout} Suppose that the balanced candy sharing game with $n$ children starts in an initial state $S$ and results in a periodic state $T$. 
Then 
\begin{enumerate}[\rm(a)]
\item[(a)]  $S$ is clockwise biased if  $0<\tau(S)<\frac{n}2$, and 
in this case $\tau(S)$  is the number of $P$s  in $T$,
\item[(b)]  $S$ is anti-clockwise biased if  $\frac{n}2<\tau(S)<n$, and 
in this case $n-\tau(S)$  is the number of $\bar P$s  in $T$.
\item[(c)]  $S$ is equivocal  if  $\tau(S)=\frac{n}2$, 
\item[(d)]  $S$ is ultimately equitable  if  $\tau(S)=0$, 
\end{enumerate}
\end{theorem}

\begin{proof}
Because of the invariance of $\tau$, we have $\tau(S)=\tau(T)$.
It remains to compute $\tau(T)$ in the four cases of Theorem~\ref{th}. 
By Remark~\ref{Re}, $\tau(I^n)=0$.
When $n$ is even, say $n=2k$, we have
\[
\tau(P^k)\equiv k+2 \sum_{i=1}^k 2i  =k+2k(k+1)   \equiv k \pmod n.
\]

Now let $T=P^{i_1} I^{j_i} P^{i_2} I^{j_2}  \dots P^{i_\ell} I^{j_\ell}$. To see that $\tau(T)=\sum_{j=1}^\ell i_j$, notice that since $\tau(I^n)=0$, we have
$\tau(T)\equiv \tau(T-I^n)\equiv  \sum_{i=1}^n ib_i \pmod n$, where 
\[
(b_1,b_2,\dots,b_n)=(-1,1)^{i_1} (0)^{j_i} (-1,1)^{i_2} (0)^{j_2}  \dots (-1,1)^{i_\ell}(0)^{j_\ell}.
\]
 Each pair $(-1,1)$ contributes 1 to the sum $\sum_{i=1}^n ib_i$.  So $\tau(T)=\sum_{j=1}^\ell i_j$ as required.

Similarly, for $T=\bar P^{i_1} I^{j_i} \bar P^{i_2} I^{j_2}  \dots \bar P^{i_\ell} I^{j_\ell}$, one has $\tau(T)\equiv - \sum_{j=1}^\ell i_j =n-\sum_{j=1}^\ell i_j$.
\end{proof}

\section{Symmetric distributions of candies.}

We conclude this study with a consideration of a special case.

\begin{theorem}\label{Thsym} Suppose that for the balanced candy sharing game with $n$ children, the initial distribution $S$ of candies is symmetrical about some diameter.
\begin{enumerate}[\rm(a)]
\item[(a)]  If  $n$ is odd, the  system is ultimately equitable.
\item[(b)]  If  $n$ is even, the  system is either ultimately equitable or equivocal, depending on whether  $\tau(S)=0$ or $\frac{n}2$ respectively.
\end{enumerate}
\end{theorem}

\begin{proof}
If the initial distribution of candies is symmetrical about some diameter, then this symmetry will persist throughout the sharing process. Of the possible periodic states given in Theorem~\ref{th}, the only symmetric ones are the fixed state $I^n$ and 
the state $P^{\frac{n}2}$ when $n$ is even. By Theorem~\ref{Thout}(b), $\tau(I^n)=0$ and $\tau(P^{\frac{n}2})=\frac{n}2$.
\end{proof}

\begin{definition}
A state in which one child has all the candies is called a \emph{monopoly state}.
\end{definition}

\begin{corollary}
In a balanced candy sharing game with $n$ candies,  monopoly states are ultimately equitable if and only if $n$ is odd.
\end{corollary}

\begin{proof}
If $n$ is odd, Theorem~\ref{Thsym} tells us that an ultimately equitable distribution of candies will be obtained. Suppose that $n$ is even,  and let $S$ be the monopoly state $S=(n,0,0,\dots,0)$. We have 
$\tau(S)\equiv \frac{n}2 +n\equiv \frac{n}2\pmod n$. 
So the result again follows from Theorem~\ref{Thsym}.
\end{proof}

\bigskip
\noindent
\emph{Thanks} 
The author is very grateful to Christian Aebi for bringing this interesting problem to his attention, carefully reading the paper, picking up typos, and making  suggestions that improved the exposition of the paper.

\end{document}